\documentclass[a4paper,12pt]{amsart} 

\usepackage[T1]{fontenc}
\usepackage[utf8]{inputenc}

\usepackage{bookmark}
\usepackage{hyperref} 

\usepackage[initials]{amsrefs}

\usepackage[%
	left=2.5cm,       
	right=2.5cm,      
	top=3.5cm,        
	bottom=3.5cm,     
	heightrounded,    
	bindingoffset=0mm 
]{geometry}

\numberwithin{equation}{section}
\newtheorem{thm}{Theorem}
\newtheorem{lem}[thm]{Lemma}
\theoremstyle{definition}
\newtheorem{rem}{Remark}

\begin{document}

\title[Nowhere bounded functions]{Sobolev subspaces  of nowhere bounded functions}

\author[P. D. Lamberti]{Pier Domenico Lamberti}
\author[G. Stefani]{Giorgio Stefani}
\email{lamberti@math.unipd.it, giorgio.stefani.2@studenti.unipd.it}
\address{Dipartimento di Matematica, Universit\`{a}  degli Studi di Padova, Via Trieste 63, 35121 Padova, Italy}


\keywords{Nowhere bounded functions, Sobolev spaces, Sobolev Embedding}

\subjclass[2010]{46E35, 26B05, 26B40}

\thanks{\emph{Acknowledgements}. 
The authors are thankful to an anonymous referee for valuable comments, in particular for considerations of terminological type. 
The authors are  thankful to Professors Richard M. Aron and Juan B. Seoane-Sep\'{u}lveda for  bringing to attention their work  with many references. The authors are members of the Gruppo Nazionale per l'Analisi Matematica, la Probabilit\`{a}  e le loro Applicazioni (GNAMPA) of the Istituto Nazionale di Alta Matematica (INdAM)}

\begin{abstract}
We prove that in any Sobolev space which is subcritical with respect to the Sobolev Embedding Theorem there exists a closed infinite dimensional linear subspace whose non zero elements are nowhere bounded functions.  We also prove the existence of a closed infinite dimensional linear subspace whose non zero elements are nowhere $L^q$ functions for suitable values of $q$ larger than the Sobolev exponent. 
\end{abstract}

\maketitle

\section{Introduction}

Given $l\in {\mathbb{N}}$, $p\in [1, \infty ]$ and an open set $\Omega$ in ${\mathbb{R}}^N$, the Sobolev space $W^{l,p}(\Omega)$ is defined as the space of those real valued functions in $L^p(\Omega)$ with distributional derivatives in $L^p(\Omega)$  up to order $l$, endowed with the  norm defined by $\| v\|_{W^{l,p}(\Omega)}=\sum_{0\le |\alpha |\le l}\| D^{\alpha }v\|_{L^p(\Omega)}$ for all $v\in W^{l,p}(\Omega)$.

Sobolev spaces play a prominent role in modern Mathematical Analysis and applications to partial differential equations. In particular, they provide a natural setting for the study of fundamental problems from Mathematical Physics, in which case the so-called energy spaces are often identified with suitable closed subspaces  of $W^{l,2}(\Omega )$ containing $C^{\infty }_c(\Omega)$, see, e.g.,~\cites{helffer, necas}.

One of the main features of Sobolev spaces is their completeness.  In fact, for $p<\infty$, they can be defined as the completion of the space of smooth functions with respect to  the norm above, which clearly allows Sobolev spaces to posses badly behaved functions.

Nevertheless, the celebrated Sobolev Embedding Theorem  gives sharp information on the intrinsic regularity of the functions in Sobolev spaces, see~\cite{burenkov}*{Chapter~4} for instance. In particular, such theorem states that if $\Omega$ is a sufficiently regular open set (say, $\Omega$ satisfies the cone condition) and $pl>N$ for $p\ne 1$, or $pl\geq N$ for $p=1$, then $W^{l,p}(\Omega)$ is continuously embedded into $C_b(\Omega)$, where $C_b(\Omega)$ denotes the space of real valued, bounded, continuous functions on $\Omega$ endowed with the usual sup-norm.  

It is well-known that, in the subcritical case $pl<N$, the Sobolev space $W^{l,p}(\Omega)$ contains unbounded functions, as well as functions which are nowhere bounded in $\Omega$. If $\Omega $ is bounded, an example of a nowhere bounded function $v$ in $W^{l,p}(\Omega)$  can be easily provided  by considering a numerable dense subset $\{x_n\}_{n\in \mathbb{N}}$ of $\Omega$ and setting $v(x)=\sum_{n=1}^{\infty }|x-x_n|^{\mu }/ 2^n$ for almost all $x\in \Omega$,  where $\mu\in ]l-N/p ,0[ $  (see also, e.g.,~\cite{evans98}*{Example~4, p.~247}; see Section~\ref{sec:proof_main} below for the general case). 

The aim of the present paper is to study such unbounded functions in the frame of a comparatively new field of investigation devoted to the analysis  of spaces of pathological  functions.  In his  seminal paper  \cite{gurariy91}, Vladimir I. Gurariy  proved the existence of  a closed infinite dimensional linear subspace of  $C([0,1])$  whose non zero elements are nowhere differentiable functions. Following \cite{gurariy91}, a number of authors   have addressed analogous problems concerning that or other counterintuitive properties of functions, see,  e.g., \cites{aronetal05, enfloetal14}. We refer to the recent monograph~\cite{aronetal15} for an extensive discussion of old and new results in this topic as well as for  references. We also refer  to~\cite{dragoetal11} for some historical and pedagogical remarks in the realm of functions with strange properties.

In this paper, we prove that every  Sobolev space $W^{l,p}(\Omega)$ with $pl\le  N$ if $p\ne 1$, or $pl<N$ if $p=1$, contains a closed infinite linear dimensional  subspace whose non zero elements are nowhere bounded functions. Actually, if $pl<N$ with $p\ge 1$, we shall prove even more. Indeed recall that, in the case $pl<N$, the Sobolev Embedding Theorem provides some additional integrability properties for the functions in $W^{l,p}(\Omega)$. More precisely, such  theorem states  that if   $\Omega$ is a sufficiently regular open set as above and  $pl<N$, then $W^{l,p}(\Omega)$ is continuously embedded into $L^{q^*}(\Omega)$, where $q^*=Np/(N-pl)$ is the celebrated Sobolev exponent. It is an exercise to prove that the Sobolev critical exponent $q^*$ cannot be improved, that is, if $pl<N$ and $W^{l,p}(\Omega)$ is continuously embedded into $L^q(\Omega)$ for some $q\in[1,\infty[$, then $q\leq q^*$. We plan to prove that this is true in a stronger way.  

Following~\cite{kaufmann&pellegrini11}, we say that a real valued function $v$ defined on $\Omega$ is nowhere $L^q$ for some $q\in[1,\infty]$ if, for any non empty open subset $U$ of $\Omega$, $v|_U$ is not in $L^q(U)$ (the case $q=\infty $ is exactly the case of nowhere bounded functions mentioned above). 

With this terminology, our main result is the following  
\begin{thm}\label{main} 
Let $\Omega$ be a non empty open set in $\mathbb{R}^N$ and let $l\in\mathbb{N}$, $p\in [1,\infty [$. If $pl<N$ then, for every $r\in]lq^*,\infty]$ fixed, the space  $W^{l,p}(\Omega)$ contains a closed infinite dimensional linear subspace whose non zero elements are  nowhere $L^r$ functions. If $pl=N$ with $p\ne 1$, then the space  $W^{l,p}(\Omega)$ contains  a closed infinite dimensional linear subspace whose non zero elements are nowhere bounded functions.
\end{thm}
 
Theorem~\ref{main} is proved in Section~\ref{sec:proof_main} where  the required closed infinite dimensional linear subspace of $W^{l,p}(\Omega)$  is defined as the image of a closed infinite dimensional linear subspace of $W^{l,p}(\Omega)\cap W^{1,pl}(\Omega)$ via a suitable compact perturbation of the identity. Such subspace  turns out to be a closed subspace also of $W^{l,p}(\Omega)\cap W^{1,pl}(\Omega)$, see Remark~\ref{fisiorem}. By the Sobolev Embedding Theorem the space $W^{1,pl}(\Omega)$ is embedded into $L^{lq^*}_{loc}(\Omega)$  and this explains why, in Theorem~\ref{main}, we require that $r$  belongs to $]lq^*, \infty ]$, which is smaller than the interval $]q^*,\infty]$ when  $l\geq 2$. We also note  that our compact operator is defined by means of a suitable composition operator and that the condition $r\in ]lq^*,\infty]$ is needed to make it well-defined when $l\geq 2$.

\section{Proof of Theorem~\ref{main}}\label{sec:proof_main}
 
In this section, we always assume that $N$, $l$, $p$ and $r$ are fixed and are as in Theorem~\ref{main}.  Note that $N\ge2$. Moreover, we  need to fix a number $a$ in $]0,1[$ as follows. If $pl<N$ and $r<\infty$, we take $a\in]0,1[$ such that $ar>lq^*$; if $pl=N$ or $r=\infty$, $a$  is any  number in the interval $]0,1[$.

We begin with some preliminaries. Let $f\in W^{l,p}(\mathbb{R}^N)$ be a function with compact support, continuous in $\mathbb{R}^N\setminus \{0\}$, that does not change sign and such that $|f(x)|\to\infty$ as $x\to 0$. In addition, for reasons that will be clear later, in the case $l\geq 2$ we also require   that $f\in W^{1,pl}({\mathbb{R}^N})$. In the case $pl<N$ and $r< \infty$, we also require the extra condition  $f\notin L^q(\mathbb{R}^N)$ for any $q\in[ar,\infty]$. The existence of such functions is well-known, see~\cite{burenkov}*{Example~8, p.~32}. 

Let $\{x_n\}_{n\in\mathbb{N}}$ be a numerable dense set in $\mathbb{R}^N$. Let $u$ be the real valued function defined in $\mathbb{R}^N$ by
\begin{equation}\label{evansfunc}
u(x)=\sum_{n=1}^{\infty}\frac{1}{2^n} f(x-x_n),
\end{equation}
for  $x\in\mathbb{R}^N$. It is an exercise to prove that the series in~\eqref{evansfunc} is convergent in $W^{l,p}(\mathbb{R}^N)$, as well as in $W^{1,pl}(\mathbb{R}^N)$ if $l\geq 2$. Moreover, such series also converges almost everywhere in $\mathbb{R}^N$. Since $f$ does not change sign, by \eqref{evansfunc} it follows that  
\begin{equation}\label{exp}
|u(x)|\ge\frac{1}{2^n}|f(x-x_n)|,
\end{equation}
for all $n\in \mathbb{N}$ almost everywhere on $\mathbb{R}^N$. Thus $u$ is nowhere bounded in $\mathbb{R}^N$ and belongs to $W^{l,p}(\mathbb{R}^N)$, and also to $W^{1,pl}(\mathbb{R}^N)$ if $l\geq 2$. Moreover, in the case $pl<N$, $u$ is also  nowhere $L^{ar}$ in $\mathbb{R}^N$.

Let $(a_n)_{n\in\mathbb{N}}\subset]a,1[$ be a strictly decreasing sequence. Let us now define a sequence of real valued functions $\{q_n\}_{n\in\mathbb{N}}$ on $\mathbb{R}$ by setting $q_n(t)=|t|^{a_n}$ for all $t\in \mathbb{R}$, $n\in\mathbb{N}$. Note that, for all $n,m\in\mathbb{N}$,
\begin{equation}
\left|\frac{d^m}{dt^m}q_n(t)\right| \le m!\, |t|^{a_n-m}, \quad \forall t\ne 0.
\end{equation}

Let $\psi\in C^\infty(\mathbb{R})$ be fixed in such a way that $\psi(t)=0$ for all $|t|\le 1$ and $\psi(t)=1$ for all $|t|\geq 2$. We set $Q_n=\psi q_n$ for all $n\in\mathbb{N}$. Clearly $Q_n\in C^\infty(\mathbb{R})$ and by the Leibniz rule it easily follows that, for every $m\in\mathbb{N}$, there exists $k_m>0$ independent of $n$ such that 
\begin{equation*}
\left|\frac{d^m}{dt^m}Q_n(t)\right|\le k_m,
\end{equation*} 
for all $t\in\mathbb{R}$, $n\in\mathbb{N}$.

Thus, for any sequence $c=(c_n)_{n\in\mathbb{N}}\in\ell^1(\mathbb{N})$, the function defined by
\begin{equation*}
g_c(t)=\sum_{n=1}^{\infty}c_nQ_n(t), \quad t\in\mathbb{R},
\end{equation*} 
belongs to $C^\infty(\mathbb{R})$ and for every $m\in\mathbb{N}$ we have
\begin{equation*}
\left|\frac{d^m}{dt^m}g_c(t)\right|=\left|\sum_{n=1}^\infty c_n\frac{d^m}{dt^m}Q_n(t)\right|\le k_m\|c\|_{\ell^1(\mathbb{N})},
\end{equation*}
for all $t\in\mathbb{R}$. Moreover, for any $c\in\ell^1(\mathbb{N})\setminus\{0\}$, we have that
\begin{equation}\label{eq:explosion_g_c}
|g_c(t)|=\left|\sum_{n\ge \bar{n}}c_n|t|^{a_n}\right|=|t|^{a_{\bar{n}}}\left|\sum_{n\ge \bar{n}}c_n|t|^{a_n-a_{\bar{n}}}\right|\ge \frac{|c_{\bar{n}}|}{2}|t|^a,
\end{equation}
for any $|t|$ sufficiently big, where $\bar{n}=\min\{n\in\mathbb{N} : c_n\ne 0\}$.

 We can now prove the following result, where $W^{l,p}(\mathbb{R}^N)\cap W^{1,pl}(\mathbb{R}^N)$ is endowed with the usual norm obtained by summing the norms of $ W^{l,p}(\mathbb{R}^N)$ and $W^{1,pl}(\mathbb{R}^N)$. We set $bc=(b_nc_n)_{n\in\mathbb{N}}\in\ell^1(\mathbb{N})$ for all $b,c\in\ell^2(\mathbb{N})$ and we note that $\|bc\|_{\ell^1(\mathbb{N})}\le\|b\|_{\ell^2(\mathbb{N})} \|c\|_{\ell^2(\mathbb{N})}$. 

\begin{lem}\label{lemma:operator_T_u} Let $u$ be defined  as in~\eqref{evansfunc} and let $b=(b_n)_{n\in\mathbb{N}}\in\ell^2(\mathbb{N})$ be fixed with $b_n\ne 0$ for all $n\in\mathbb{N}$. 
The linear operator ${\mathcal{T}}_u$ from $\ell^2(\mathbb{N})$ to $W^{l,p}(\mathbb{R}^N)\cap W^{1,pl}(\mathbb{R}^N)$ defined by 
\begin{align*}
{\mathcal{T}}_u(c)=g_{bc}\circ u =\left(\sum_{n=1}^{\infty }b_nc_nQ_n\right)\circ u,
\end{align*}
for all $c=(c_n)_{n\in\mathbb{N}}\in\ell^2(\mathbb{N})$,
is continuous, injective and compact. Moreover, ${\mathcal{T}}_u(c)$ is a nowhere bounded function in $\mathbb{R}^N$ for all $c\in\ell^2(\mathbb{N})\setminus\{0\}$ and, in the case $pl<N$, ${\mathcal{T}}_u(c)$ is a  nowhere $L^r$ function in $\mathbb{R}^N$ for all $c\in\ell^2(\mathbb{N})\setminus\{0\}$.
\end{lem}

\begin{proof} 
Since $g_{bc}\in C^\infty(\mathbb{R})$ has bounded derivatives, $g_{bc}(0)=0$, $u\in W^{l,p}(\mathbb{R}^N)$ if $l\geq 1$ and $u\in W^{l,p}(\mathbb{R}^N)\cap W^{1,pl}(\mathbb{R}^N)$ if $l\geq 2$, the function ${\mathcal{T}}_u(c)$ belongs to $W^{l,p}(\mathbb{R}^N)\cap W^{1,pl}(\mathbb{R}^N)$. Indeed, the case $l=1$ is a direct application of the chain rule, which also allows to easily prove that 
\begin{equation}\label{chain}
\|{\mathcal{T}}_u(c)\|_{W^{1,p}(\mathbb{R}^N)}\leq k_1\|b\|_{\ell^2(\mathbb{N})} \|c\|_{\ell^2(\mathbb{N})} \|u\|_{W^{1,p}(\mathbb{R}^N)}.
\end{equation}
In the case $l\geq 2$, one needs to use the condition $u\in W^{l,p}(\mathbb{R}^N)\cap W^{1,pl}(\mathbb{R}^N)$ in a substantial way, in order to avoid the possible appearance of the so called \emph{Dahlberg degeneracy phenomenon}, which prevents non trivial composition operators to preserve Sobolev spaces $W^{l,p}$ when $1+1/p<l<N/p$ (including those with fractional order of smoothness). For further details and discussions, we refer to~\cite{runst&sickel96}*{\S5.2.5} and~\cite{dahlberg}. In particular, by estimate~(2) in~\cite{runst&sickel96}*{\S5.2.5}, we can also immediately deduce that there exists a constant $K>0$, independent of $u$, $b$, and $c$, such that 
\begin{equation}\label{wingfriedest}
\|{\mathcal{T}}_u(c)\|_{W^{l,p}(\mathbb{R}^N)}\leq K\max_{i=1,\dots , l} k_i\ \|b\|_{\ell^2(\mathbb{N})} \|c\|_{\ell^2(\mathbb{N})} \left( \|u\|_{W^{l,p}(\mathbb{R}^N)}+ \|u\|^l_{W^{1,pl}(\mathbb{R}^N)}\right).
\end{equation}
 When $l\ge2$, similarly to~\eqref{chain}, we also have
\begin{equation}\label{chain_pl}
\|{\mathcal{T}}_u(c)\|_{W^{1,pl}(\mathbb{R}^N)}\leq k_1\|b\|_{\ell^2(\mathbb{N})} \|c\|_{\ell^2(\mathbb{N})} \|u\|_{W^{1,pl}(\mathbb{R}^N)}.
\end{equation}
By~\eqref{chain}, \eqref{wingfriedest} and~\eqref{chain_pl}, we can conclude that the operator ${\mathcal{T}}_u$ is well defined and continuous.

By~\eqref{exp} combined with~\eqref{eq:explosion_g_c}, we have  that for any $c\in \ell^2(\mathbb{N})\setminus \{0\}$ and $n\in\mathbb{N}$ there exists an open neighbourhood $U_n$
of $x_n$ and $\alpha_n >0$ such that 
\begin{equation}\label{belowest}
|{\mathcal{T}}_u(c)|\geq \alpha_n  |u|^a \text{ on } U_n.
\end{equation}
 By~\eqref{belowest}, it follows that ${\mathcal{T}}_u(c)$ is nowhere bounded in $\mathbb{R}^N$  and, if $pl<N$, ${\mathcal{T}}_u(c)$ is also nowhere $L^r$ in $\mathbb{R}^N$ for all $c\in\ell^2(\mathbb{N})\setminus\{0\}$. In particular, the operator ${\mathcal{T}}_u$ is injective.

It remains to prove that ${\mathcal{T}}_u$ is compact. For any $k\in\mathbb{N}$, consider the linear continuous operator $\mathcal{T}^{(k)}_u$ from $\ell^2(\mathbb{N})$ to $W^{l,p}(\mathbb{R}^N)\cap W^{1,pl}(\mathbb{R}^N)$ defined by 
\begin{equation*}
{\mathcal{T}}^{(k)}_u(c)=\left(\sum_{n=1}^k b_nc_n Q_n\right)\circ u,
\end{equation*}
for all $c\in\ell^2(\mathbb{N})$. Note that ${\mathcal{T}}_u^{(k)}$ is a finite rank operator. Then, by estimates analogous to~\eqref{chain}, \eqref{wingfriedest} and~\eqref{chain_pl} and by the Cauchy--Schwarz inequality, we get
\begin{equation*}
\|{\mathcal{T}}_u(c)-{\mathcal{T}}_u^{(k)}(c)\|_{W^{l,p}(\mathbb{R}^N)\cap W^{1,pl}(\mathbb{R}^N)}\leq B\left(\sum_{n=k+1}^{+\infty}b_n^2\right)^{\frac{1}{2}}\| c \|_{\ell^2(\mathbb{N})},
\end{equation*}
for all $c\in \ell^2(\mathbb{N})$,
where $B>0$ is a constant independent of $c$. Hence
\begin{equation*}
\lim_{k\to+\infty}\|{\mathcal{T}}_u-{\mathcal{T}}_u^{(k)}\|_{\ell^2(\mathbb{N})\to W^{l,p}(\mathbb{R}^N)\cap W^{1,pl}(\mathbb{R}^N)}=0,
\end{equation*}
where in the left-hand side we use the standard operator norm of ${\mathcal{T}}_u-{\mathcal{T}}_u^{(k)}$. Thus ${\mathcal{T}}_u$ is approximated in norm by finite rank operators, hence it is compact. This concludes the proof of Lemma~\ref{lemma:operator_T_u}.
\end{proof}

We can now prove our main result.

\begin{proof}[Proof of Theorem~\ref{main}]
Let $\bar{x}\in\Omega$ and let $R>0$ be such that $B(\bar{x},R)\subset\Omega $. Let $A$ be the open annulus defined by $A=\{x\in \mathbb{R}^N: R/2<|x-\bar x|<R\}$. We denote by $W^{l,p}_0(A)$ the standard Sobolev space defined as the closure of $C^{\infty}_c(A)$ in $W^{l,p}(A)$. We set
\begin{equation}\label{ics}
X=\{v\in W^{l,p}_0(A):\ v\text{ is radial with respect to } \bar{x}\},
\end{equation}
where it is meant that a function $v$ is radial with respect to $\bar{x}$ if the value of $v(x)$ depends only on $|x-\bar{x}|$ for all $x\in A$. The space $X$ can be naturally seen as a subspace of $W^{l,p}(\Omega)$ by extending functions by zero outside $A$. Moreover, it is straightforward that $X$ is a closed subspace of $W^{l,p}(\Omega)$ with infinite dimension. 

Since any function $v$ in $X$ is of the form $v(x)= g_v(|x-\bar x|)$ for a suitable function $g_v$ and is zero outside $A$, by the Radial Lemma (see~\cite{lions82}*{Lemma II.1}) it is easy to see that $X$ is continuously embedded into the space $C_b(R/2,R)$ via the embedding $v\mapsto g_v$. In particular, any function in $X$ is bounded and continuous.  Since $C_b(R/2,R)$ is continuously embedded into $L^2(R/2,R)$, which is a Hilbert space isometric to $\ell^2(\mathbb{N})$, we conclude that there exists a continuous embedding ${\mathcal{J}}$ of $X$ into $\ell^2(\mathbb{N})$. 

Let ${\mathcal{R}}$ be the restriction operator from $W^{l,p}(\mathbb{R}^N)$ to $W^{l,p}(\Omega)$ and let ${\mathcal{T}}_u$ be the operator defined in Lemma~\ref{lemma:operator_T_u}  considered as an operator from $\ell^2(\mathbb{N})$ to $W^{l,p}(\mathbb{R}^N)$. It is obvious that the operator ${\mathcal{R}}\circ {\mathcal{T}}_u\circ \mathcal{J}$ is injective because all the non zero functions in the image of the operator ${\mathcal{T}}_u\circ {\mathcal{J}}$ are nowhere bounded. Thus ${\mathcal{R}}\circ {\mathcal{T}}_u\circ {\mathcal{J}}$ is a compact embedding of $X$ into $W^{l,p}(\Omega)$. We now consider the operator ${\mathcal{I}}-{\mathcal{R}}\circ {\mathcal{T}}_u\circ {\mathcal{J}}$ from $X$ to $W^{l,p}(\Omega)$, where ${\mathcal{I}}$ denotes the identity operator of $W^{l,p}(\Omega)$. 

We note that $\ker({\mathcal{I}}-{\mathcal{R}}\circ {\mathcal{T}}_u\circ {\mathcal{J}})=\{0 \}$. Indeed, if $v\in X$ is such that $v={\mathcal{R}}\circ {\mathcal{T}}_u\circ {\mathcal{J}}(v)$, then ${\mathcal{R}}\circ {\mathcal{T}}_u\circ {\mathcal{J}}(v)$ is a bounded continuous function; but this implies that ${\mathcal{J}}(v)=0$, hence $v=0$ (in fact, otherwise, if $v\ne 0$ then ${\mathcal{R}}\circ {\mathcal{T}}_u\circ {\mathcal{J}}(v)$ would be nowhere bounded).

Let $Y$ be the subspace of $W^{l,p}(\Omega)$ defined by 
\begin{equation}\label{ipsilon}
Y=({\mathcal{I}}-{\mathcal{R}}\circ {\mathcal{T}}_u\circ\mathcal{J})(X).
\end{equation}
Since ${\mathcal{I}}-{\mathcal{R}}\circ {\mathcal{T}}_u\circ\mathcal{J}$ is a compact perturbation of the identity, by the Fredholm Alternative Theorem (see, e.g.,~\cite{brezis}*{Exercise~6.9, (4)})  it follows that $Y$ is a closed subspace of $W^{l,p}(\Omega)$. In particular, since ${\mathcal{I}}-{\mathcal{R}}\circ {\mathcal{T}}_u\circ\mathcal{J}$ is injective, $\dim Y =\dim X=\infty$. 

 Finally, we observe that any function in $ Y\setminus\{0\}$ is nowhere bounded. Indeed, as observed above, any function $v\in X$ is bounded, while ${\mathcal{R}}\circ {\mathcal{T}}_u\circ\mathcal{J}(v)$ is nowhere bounded if $v\ne 0$. Hence $({\mathcal{I}}-{\mathcal{R}}\circ {\mathcal{T}}_u\circ\mathcal{J})(v)$ is nowhere bounded for all $v\in X\setminus\{0\}$. Moreover, in the case $pl<N$, $({\mathcal{I}}-\mathcal{R}\circ {\mathcal{T}}_u\circ\mathcal{J})(v)$ is nowhere $L^r$ for all $v\in X\setminus\{0\}$, because $\mathcal{R}\circ\mathcal{T}_u\circ\mathcal{J}(v)$ is nowhere $L^r$ for all $v\in X$. This concludes the proof.
\end{proof}

\begin{rem}\label{fisiorem} Assume that $l\geq 2$. In the proof of Theorem~\ref{main} one can replace the space $W^{l,p}$ by $W^{l,p}\cap W^{1,pl}$. Indeed, a simple argument allows to see that the space $X$ in \eqref{ics} is a closed subspace also of $W^{l,p}(\Omega)\cap W^{1,pl}(\Omega )$. Thus, by Lemma~\ref{lemma:operator_T_u}, one can consider 
the operator  ${\mathcal{T}}$ in the proof of Theorem~\ref{main} as an operator from $X$ (considered as a closed subspace of $W^{l,p}(\Omega)\cap W^{1,pl}(\Omega )$) to $W^{l,p}(\Omega)\cap W^{1,pl}(\Omega )$, and conclude that the space $Y$ in \eqref{ipsilon} is also  a closed subspace also of $W^{l,p}(\Omega)\cap W^{1,pl}(\Omega )$.
\end{rem}

\begin{rem}
We briefly observe that the subcritical Sobolev space $W^{l,p}(\Omega)$ contains also a $\mathfrak{c}$-dimensional linear subspace and a countably generated algebra whose non zero elements are nowhere bounded functions.

The existence of such a $\mathfrak{c}$-dimensional linear subspace follows immediately from the fact that any infinite dimensional Banach space has Hamel dimension at least~$\mathfrak{c}$, see~\cite{lacey73}. Thus, in Theorem~\ref{main}, it is possible to replace `closed infinite dimensional linear subspace' with `$\mathfrak{c}$-dimensional linear subspace'. For a simpler construction, consider the vector space generated by the functions $h_\mu\circ u$, where $h_\mu$ is a real valued function defined on $\mathbb{R}$ such that $h_\mu\in C^l(\mathbb{R})$, $h_\mu(0)=0$ and $h_\mu(t)=|t|^\mu$ for $|t|\ge 1$ for any $\mu\in[a,1]$, where $u$ is the function defined in~\eqref{evansfunc} and $a\in]0,1[$ is chosen as at the beginning of Section~\ref{sec:proof_main}.  

To prove the existence of a countably generated algebra of nowhere bounded functions, it is enough to consider the algebra generated by the functions $l_n\circ u$, where $l_n$ is a real valued function defined on $\mathbb{R}$ such that $l_n\in C^l(\mathbb{R})$, $l_n(0)=0$ and $l_n(t)=\log(1+l_{n-1}(t))$, $l_0(t)=t$, for $|t|\ge 1$ and $n\in\mathbb{N}$. Therefore, in the case $pl\le N$ for $p\ne 1$, or $pl<N$ for $p=1$, $W^{l,p}(\Omega)$ contains a countably generated algebra whose non zero elements are nowhere bounded functions. 
\end{rem}


\end{document}